\documentclass[11pt,a4paper]{article}
\usepackage{fullpage}

\usepackage[utf8]{inputenc}
\usepackage{amsthm,amssymb}
\usepackage[leqno]{amsmath}
\usepackage{tcolorbox}
\usepackage{thmtools}
\usepackage{subcaption,thm-restate}
\usepackage[mathcal,mathscr]{eucal}
\usepackage{epsfig,graphicx,graphics,color}
\usepackage{caption}
\usepackage{enumerate}
\usepackage{enumitem, crossreftools}
\usepackage[sort]{cite}
\usepackage{titling}
\usepackage{hyperref}
\usepackage{tikz}
\hypersetup{
    colorlinks=true,       
    linkcolor=blue,        
    citecolor=red,         
    filecolor=magenta,     
    urlcolor=cyan,         
    linktocpage=true
}

\theoremstyle{plain}
\newtheorem{thm}{Theorem}

\newtheorem{prop}[thm]{Proposition}

\theoremstyle{definition}

\newtheorem{rem}[thm]{Remark}

\def\final{0}  
\def\iflong{\iffalse}
\ifnum\final=0  
\newcommand{\knote}[1]{{\color{red}[{\tiny \textbf{Kristóf:} \bf #1}]\marginpar{\color{red}*}}}
\newcommand{\yanote}[1]{{\color{red}[{\tiny \textbf{Yutaro:} \bf #1}]\marginpar{\color{red}*}}}
\newcommand{\yonote}[1]{{\color{red}[{\tiny \textbf{Yu:} \bf #1}]\marginpar{\color{red}*}}}
\newcommand{\ktnote}[1]{{\color{red}[{\tiny \textbf{K. Tamás:} \bf #1}]\marginpar{\color{red}*}}}
\newcommand{\stnote}[1]{{\color{red}[{\tiny \textbf{S. Tamás:} \bf #1}]\marginpar{\color{red}*}}}
\else 
\newcommand{\knote}[1]{}
\newcommand{\yanote}[1]{}
\newcommand{\yonote}[1]{}
\newcommand{\ktnote}[1]{}
\newcommand{\stnote}[1]{}
\fi

\DeclareMathOperator\clo{cl}

\newcommand{\bR}{\mathbb{R}}
\newcommand{\cI}{\mathcal{I}}
\newcommand{\cF}{\mathcal{F}}
\newcommand{\cH}{\mathcal{H}}

\newcommand{\cZ}{\mathcal{Z}}

\title{Hypergraph characterization of split matroids}

\thanksmarkseries{alph}
\author{
\hfill
\and
Kristóf Bérczi\thanks{MTA-ELTE Momentum Matroid Optimization Research Group, Budapest, Hungary.} \thanks{MTA-ELTE Egerváry Research Group, Budapest, Hungary.} \thanks{Department of Operations Research, Eötvös Loránd University, Budapest, Hungary. Email: \texttt{kristof.berczi@ttk.elte.hu, tamas.kiraly@ttk.elte.hu, tamas.schwarcz@ttk.elte.hu}.}
\and
Tamás Király\footnotemark[1] \footnotemark[2] \footnotemark[3]
\and
Tamás Schwarcz\footnotemark[1] \footnotemark[3]
\and
\hfill
\and
Yutaro Yamaguchi\thanks{Department of Information and Physical Sciences, Graduate School of Information Science and Technology, Osaka University, Osaka, Japan. Email: \texttt{yutaro.yamaguchi@ist.osaka-u.ac.jp}.}
\and
Yu Yokoi\thanks{Principles of Informatics Research Division, National Institute of Informatics, Tokyo, Japan. Email: \texttt{yokoi@nii.ac.jp}.}
}

\begin{document}
\maketitle

\begin{abstract}
We provide a combinatorial study of split matroids, a class that was motivated by the study of matroid polytopes from a tropical geometry point of view. A nice feature of split matroids is that they generalize paving matroids, while being closed under duality and taking minors. Furthermore, these matroids proved to be useful in giving exact asymptotic bounds for the dimension of the Dressian, and also implied new results on the rays of the tropical Grassmannians.  

In the present paper, we introduce the notion of elementary split matroids, a subclass of split matroids that contains all connected split matroids. We give a hypergraph characterization of elementary split matroids in terms of independent sets, and show that the proposed class is closed not only under duality and taking minors but also truncation. We further show that, in contrast to split matroids, the proposed class can be characterized by a single forbidden minor. As an application, we provide a complete list of binary split matroids.  

\medskip

\end{abstract}

\section{Introduction}
The class of split matroids was introduced by Joswig and Schr\"oter~\cite{joswig2017matroids} as an efficient tool in tropical geometry. Their definition was based on a polyhedral approach, by imposing conditions on the split hyperplanes of the matroid base polytope. In order to recall the definition, we first overview the polyhedral background.

Given a polytope $P$, its intersection with a supporting hyperplane is called a \textbf{face} of $P$. The polytope itself is also considered to be a face. A face is a \textbf{facet} if it is properly contained in exactly one face, namely $P$. A \textbf{split} of $P$ is a subdivision without new vertices which has exactly two maximal cells. The affine span of the intersection of the two cells is then called a \textbf{split hyperplane}. Two splits are \textbf{compatible} if the corresponding split hyperplanes do not meet in a relative interior point of $P$. Let $M=(S,r_M)$ be a matroid on ground set $S$ with rank function $r_M$. We denote the rank of the matroid by $r$, that is, $r_M(S)=r$. The convex hull of the characteristic vectors of the bases of $M$ is called the \textbf{matroid base polytope} of $M$ and is denoted by $P(M)$. We denote by $\Delta(r,S)$ the $|S|-1$-dimensional hypersimplex representing the matroid base polytope of the rank-$r$ uniform matroid on $S$, that is, the convex hull of all zero-one vectors over $S$ with exactly $r$ ones. For a flat $F$ of $M$, the \textbf{$F$-hyperplane} is defined as $H(F)=\{x\in\bR^S\mid x(F)=r_M(F)\}$, while the intersection $P(M)\cap H(F)$ is the \textbf{face of $P(M)$ defined by $F$}. Note that two flats might define the same face. A flat $F$ is called a \textbf{flacet} if it defines a facet and is inclusionwise minimal among flats defining $H(F)$. A flacet $F$ is a \textbf{split flacet} if the corresponding $F$-hyperplane defines a split of $\Delta(r,S)$. Roughly, the split flacets are the hyperplanes that are used to cut off parts of $\Delta(r,S)$ to obtain $P(M)$. Using this terminology, a matroid $M$ is a \textbf{split matroid} if its split flacets form a compatible system of splits of the affine hull of $P(M)$ intersected with the unit cube $[0,1]^S$.

The goal of the present paper is to give a combinatorial understanding of split matroids.

\paragraph{Previous work.}
Joswig and Schr\"oter~\cite{joswig2017matroids} gave a thorough analysis of split matroids in terms of polyhedral geometry. They observed that it suffices to concentrate on the connected case, as a matroid is a split matroid if and only if at most one connected component is a non-uniform split matroid and all other components are uniform~\cite[Proposition 15]{joswig2017matroids}. For the connected case, they also gave a characterization that does not rely on polyhedral combinatorics, stating that a connected matroid is a split matroid if and only if for each split flacet $F$ the restriction $M|F$ and the contraction $M/F$ both are uniform~\cite[Theorem 11]{joswig2017matroids}. 

Besides their applicability in tropical geometry, split matroids are also of combinatorial interest. In particular, the class of split matroids contains all paving matroids, a well-studied class with distinguished structural properties. A conjecture of Crapo and Rota~\cite{crapo1970foundations}, that was made precise by Mayhew, Newman, Welsh and Whittle~\cite{mayhew2011asymptotic}, suggests that the asymptotic fraction of matroids on $n$ elements that are paving tends to $1$ as $n$ tends to infinity. Therefore, an affirmative answer to the conjecture would imply that almost all matroids are split. A weakness of paving matroids is that their class is not closed under duality, a property that is desired in many cases. However, split matroids are closed both under duality and taking minors~\cite[Proposition 44]{joswig2017matroids}, hence they form a large class with strong combinatorial properties.   

In the light of minor-closedness, it is natural to ask what the excluded minors are for the class of split matroids. It is an easy exercise to show that uniform matroids are exactly the $U_{0,1}\oplus U_{1,1}$ minor-free matroids. The broader class of paving matroids coincides with the family of $U_{0,1}\oplus U_{2,2}$-minor-free matroids~\cite{oxley1991ternary}. For split matroids, Joswig and Schr\"oter~\cite[Question A]{joswig2017matroids} identified five forbidden minors, and Cameron and Mayhew~\cite{cameron2021excluded} later verified that the list is complete. It is worth mentioning that $U_{0,\ell}\oplus U_{k,k}$-minor-free matroids were studied for positive integers $k$ and $\ell$ in general~\cite{drummond2021generalisation}, while \cite{vega2020thesis} gave excluded-minor characterizations for the class of so-called nearly-uniform and nearly-paving matroids.

\paragraph{Our results.}
In \cite{joswig2017matroids}, split matroids were introduced via polyhedral geometry. The polyhedral point of view gives an insight into the geometry of the base polytope of split matroids, which in turn leads to a series of fundamental structural results. However, the polyhedral approach has two shortcomings when it comes to optimization. First, the definition is difficult to work with as it relies on the joint structure of split flacets. For example, it is not immediate to see what the independent sets are, or how the rank of a set can be determined. Furthermore, as it was already observed in \cite{joswig2017matroids}, the notion of split matroids is a bit subtle in the disconnected case. This is strengthened by the fact that while uniform or paving matroids can be characterized by a single excluded minor, the class of split matroids requires five of those. 

The above observations suggest that there might be an intermediate matroid class that captures all the good characteristics of split matroids (i.e.\ closed under duality and taking minors) but is more convenient to work with in terms of optimization. We show that this is indeed true and introduce a class that we call elementary split matroids. The proposed class is a proper subclass of split matroids which includes all connected split matroids. The definition follows a combinatorial approach by setting the independent sets of the matroid to be the family of sets having bounded intersections with certain hyperedges. An analogous characterization was previously known for paving matroids, see \cite{hartmanis1959lattice,welsh1976matroid}: for a non-negative integer $r$, a ground set $S$ of size at least $r$, and a (possibly empty) family $\mathcal{H}=\{H_1,\dots,H_q\}$ of proper subsets of $S$ such that $|H_i\cap H_j|\leq r-2$ for $1\leq i < j\leq q$, the set system $\mathcal{B}_{\mathcal{H}}=\{X\subseteq S\mid |X|=r,\ X\not\subseteq H_i\ \text{for } i=1,\dots,q\}$ forms the set of bases of a paving matroid, and in fact every paving matroid can be obtained in this form. Elementary split matroids satisfy similar constraints; nevertheless, the underlying hypergraph might have more complex structure. 

We show that the proposed class has various nice properties that partially follow from connected split matroids being special cases. However, elementary split matroids are closed not only under duality and taking minors but also truncation. Furthermore, the class can be characterized by the single forbidden minor $U_{0,1}\oplus U_{1,2}\oplus U_{1,1}$, therefore fitting in the list of earlier results. Based on the excluded-minor characterization, we give a new proof for the result of \cite{cameron2021excluded} on forbidden minors of split matroids and a complete description of binary split matroids.

\medskip

The rest of the paper is organized as follows. Basic notation and definitions are introduced in Section~\ref{sec:pre}. The hypergraph representation of elementary split matroids is presented in Section~\ref{sec:hyp}. Section~\ref{sec:minor} gives an excluded-minor characterization of the proposed class. Finally, Section~\ref{sec:app} gives a new proof for the list of forbidden minors of split matroids, and further provides a complete list of binary split matroids.

\section{Preliminaries}
\label{sec:pre}

Let $S$ be a ground set of size $n$. A \textbf{clutter} (or \textbf{Sperner family)} is a collection $\cF$ of subsets of $S$ in which none of the sets is a subset of another. For subsets $X,Y\subseteq S$, 
the \textbf{difference} of $X$ and $Y$ is denoted by $X-Y$. If $Y$ consists of a single element $y$, then $X-\{y\}$ and $X\cup\{y\}$ are abbreviated as $X-y$ and $X+y$, respectively.

A \textbf{matroid} is a pair $M=(S,\cI)$ where $\cI\subseteq 2^S$ is the family of \textbf{independent sets} satisfying the so-called \textbf{independence axioms}: 
\begin{enumerate}[start=1, label={(I\arabic*)}] \itemsep0em
    \item \label{eq:I1} $\emptyset\in\cI$,
    \item \label{eq:I2} if $X\subseteq Y$ and  $Y\in\cI$, then $X\in\cI$,
    \item \label{eq:I3} for every subset $X\subseteq S$ the maximal subsets of $X$ which are in $\cI$ have the same cardinality. 
\end{enumerate}
For a set $X\subseteq S$, the maximum size of an independent subset of $X$ is the \textbf{rank} of $X$ and is denoted by $r_M(X)$. The subscript $M$ is dismissed when the matroid is clear from the context. The inclusionwise maximal members of $\cI$ are called \textbf{bases}. An inclusionwise minimal non-independent set forms a \textbf{circuit}, while a \textbf{loop} is a circuit consisting of a single element. 
The \textbf{dual} of $M$ is the matroid $M^*=(S,\cI^*)$ where $\cI^*=\{X\subseteq S \mid S-X\ \text{contains a basis of $M$}\}$. 
A \textbf{cocircuit} or \textbf{coloop} of $M$ is a circuit or loop of $M^*$, respectively. 
The matroid is \textbf{connected} if for any two elements $e,f\in S$ there exists a circuit containing both. This can be shown to be equivalent to $r_M(X)+r_M(S-X)> r_M(S)$ for every $\emptyset\neq X\subsetneq S$.
A set $X\subseteq S$ is \textbf{closed} or is a \textbf{flat} if $r_M(X+e)>r_M(X)$ for every $e\in S-X$.
The \textbf{closure} of a set $X\subseteq S$, that is, the inclusionwise minimal closed set containing $X$ is denoted by $\clo_M(X)$.
Two non-loop elements $e,f\in S$ are \textbf{parallel} if $r_M(\{e,f\})=1$. A flat of rank one is called a \textbf{parallel class}. 
A flat is \textbf{proper} if it has nonzero rank and it is not the ground set of the matroid. 
A subset $Z\subseteq S$ is \textbf{cyclic} if it is the (possibly empty) union of circuits, or equivalently, the matroid restricted to $Z$ has no coloops. Bonin and de~Mier \cite{bonin2008lattice} rediscovered the following axiom scheme for the cyclic flats of a matroid, first proved by Sims \cite{sims1980thesis}.

\begin{prop} \label{prop:cyclic}
Let $\cZ$ be a collection of subsets of a ground set $S$ and $r\colon \cZ \to \mathbb{Z}_{\ge 0}$ a function. There is a matroid $M$ on $S$ for which $\cZ$ is the set of cyclic flats and $r$ is the rank function of $M$ restricted to $\cZ$ if and only if the following conditions are satisfied:
\begin{enumerate}[start=0, label={(Z\arabic*)}] \itemsep0em
    \item \label{eq:Z0}  $\cZ$ is a lattice under inclusion,
    \item \label{eq:Z1} $r(0_\cZ) = 0$ where $0_\cZ$ is the zero of this lattice,
    \item \label{eq:Z2} $0<r(Y)-r(X)<|Y-X|$ for all $X, Y \in \mathcal{Z}$ with $X\subsetneq Y$,
    \item \label{eq:Z3} $r(X)+r(Y) \ge r(X \vee Y) + r(X\wedge Y) + |(X\cap Y)-(X\wedge Y)|$ for all $X,Y \in \cZ$ with join $X \vee Y$ and meet $X\wedge Y$. 
\end{enumerate}
In this case, the independent sets of $M$ are $\cI = \{I \subseteq S\mid |I \cap Z| \le r(Z) \text{ for each } Z \in \cZ\}$.
\end{prop}

For a non-negative integer $r\leq n$, the \textbf{uniform matroid} $U_{r,n}$ is defined on an $n$-element set by setting every subset of size at most $r$ to be independent, that is, $\cI=\{X\subseteq S\mid |X|\leq r\}$. When $r=n$, the matroid is called a \textbf{free matroid}, while the choice $r=0$ results in a \textbf{rank-0 matroid}. For technical reasons, we allow the ground set of the matroid to be the empty set, i.e.\ $n=0$, in which case the matroid is simply the \textbf{empty matroid} $M=(\emptyset, \{\emptyset\})$. A matroid of rank $r$ is called \textbf{paving} if every set of size at most $r-1$ is independent, or in other words, every circuit of the matroid has size at least $r$. 

The \textbf{direct sum} $M_1\oplus M_2$ of matroids $M_1=(S_1,\cI_1)$ and $M_2=(S_2,\cI_2)$ on disjoint ground sets is the matroid $M=(S_1\cup S_2,\mathcal{I})$ whose independent sets are the disjoint unions of an independent set of $M_1$ and an independent set of $M_2$, that is, $\cI=\{I_1 \cup I_2 \mid I \in \cI_1 \text{ and } I_2 \in \cI_2\}$.
Given a non-negative integer $k$, the \textbf{$k$-truncation} of $M=(S,\cI)$ is the matroid $(M)_{k}=(S,\cI_k)$ with $\cI_k=\{X\in\cI\mid |X|\leq k\}$. 
Given a subset $S'\subseteq S$, the \textbf{restriction} of $M$ to $S'$ is again a matroid $M|S'=(S',\cI')$ with independence family $\cI'=\{I\in\cI\mid I\subseteq S'\}$. We also say that $M|S'$ is obtained by the \textbf{deletion} of $S-S'$, denoted by $M\backslash (S-S')$. The \textbf{contraction} of a subset $S''\subseteq S$ results in a matroid $M/S''=(S-S'',\cI'')$ where $\cI''=\{I\in\cI\mid I\subseteq S-S'',\ |I|=r_M(S''\cup I)-r_M(S'')\}$. 
A matroid $N$ that can be obtained from $M$ by a sequence of deletions and contractions is called a \textbf{minor} of $M$. For uniform matroids, it is not difficult to see that $U_{k',\ell'}$ is a minor of $U_{k,\ell}$ if and only if $k'\leq k$ and $\ell-\ell'\geq k-k'$ hold.
The following well-known result is \cite[Theorem~4.3.1]{oxley2011matroid}.

\begin{prop} \label{prop:connected}
Let $e$ be an element of a connected matroid $M$. Then $M/e$ or $M\backslash e$ is connected.
\end{prop}

A class $\mathcal{M}$ of matroids is \textbf{minor-closed} if for any member $M$ of $\mathcal{M}$, each minor of $M$ is also contained in $\mathcal{M}$. 
For a minor-closed class $\mathcal{M}$, a \textbf{nearly-$\mathcal{M}$ matroid} is a matroid $M$ such that $M/e \in \mathcal{M}$ or $M\backslash e\in \mathcal{M}$ for each element $e$. We will use the following observation of \cite{tejeda2020thesis, vega2020thesis}. 

\begin{prop} \label{prop:nearly}
The class of nearly-$\mathcal{M}$ matroids is minor-closed for each minor-closed class $\mathcal{M}$ of matroids.
\end{prop}

The \textbf{rank-2 wheel} $M(\mathcal{W}_2)$ is the matroid obtained from $U_{2,3}$ by adding a parallel copy of one of the elements of the ground set.
The following is a consequence of a result of Gershkoff and Oxley \cite{gershkoff2018notion}.
\begin{prop} \label{prop:w2}
Every connected non-uniform matroid contains $M(\mathcal{W}_2)$ as a minor.
\end{prop}

For connected matroids, the following proposition summarizes the relations between the different notions of flacets and characterizes compatibility, see \cite{cameron2021excluded,joswig2017matroids,fujishige1984characterization,feichtner2005}.

\begin{prop} Let $M$ be a connected matroid on ground set $S$ with rank function $r_M$.
\begin{enumerate}[label=(\alph*)]\itemsep0em \label{prop:flacet}
    \item A subset $Z\subseteq S$ is a flacet of $M$ if and only if it is a proper flat such that both $M|Z$ and $M/Z$ are connected. \label{it:1}
    \item A flacet $Z$ is a split flacet if and only if $|Z|\ge 2$, or equivalently, if $Z$ is cyclic. \label{it:2} 
    \item For distinct split flacets $F$ and $G$, the splits obtained from the $F$- and $G$-hyperplanes are compatible if and only if $|F\cap G|+r_M(S) \le r_M(F)+r_M(G)$. \label{it:3}
\end{enumerate}
\end{prop}
As we will show, the inequality of \ref{it:3} motivates a matroid class slightly different from that of split matroids. 

\section{Hypergraph representation}
\label{sec:hyp}

In this section, we introduce the notion of elementary split matroids. Similarly to paving matroids, the definition is via hypergraphs, which will immediately imply that the proposed class is closed under duality, taking minors, and truncation. 

\begin{thm}\label{thm:hyp}
Let $S$ be a ground set of size at least $r$, $\cH=\{H_1,\dots, H_q\}$ be a (possibly empty) collection of subsets of $S$, and $r, r_1, \dots, r_q$ be non-negative integers satisfying
\begin{equation}
|H_i \cap H_j| \le r_i + r_j -r\ \text{for $1 \le i < j \le q$}. \tag*{(H1)}\label{eq:h1}
\end{equation}
Then $\cI=\{X\subseteq S\mid |X|\leq r,\ |X\cap H_i|\leq r_i\ \text{for $1\leq i \leq q$}\}$ forms the independent sets of a matroid with rank function $r_M(Z)=\min\big\{r,|Z|,\displaystyle\min_{1\leq i\leq q}\{|Z-H_i|+r_i\}\big\}$. If furthermore
\begin{equation}
|S-H_i| + r_i \ge r\ \text{for $i=1,\dots, q$} \tag*{(H2)}\label{eq:h2}
\end{equation}
holds, then the rank of the matroid is $r$.
\end{thm}
\begin{proof}
The first two independence axioms clearly hold. A nice trick of the proof is that the third independence axiom \ref{eq:I3} and the rank formula is proved simultaneously. For any subset $Z\subseteq S$, let $I\subseteq Z$ be a maximal member of $\cI$ in the sense that $I$ cannot be extended by an element of $Z$ to a member of $\cI$. If $|I|=\min\{r,|Z|\}$ then we are done, hence assume that strict inequality holds. Since $I$ is maximal in $Z$, there is a hyperedge $H_z\in\cH$ for every $z\in Z-I$ such that $|I\cap H_z|=r_z$ and $z\in H_z$. Furthermore, if $z',z''\in Z-I$ are distinct elements, then the corresponding hyperedges $H_{z'}$ and $H_{z''}$ are identical as otherwise 
\begin{align*}
|H_{z'}\cap H_{z''}|
{}&{}\geq
|I\cap H_{z'}\cap H_{z''}|\\
{}&{}=
|I\cap H_{z'}|+|I\cap H_{z''}|-|I\cap (H_{z'}\cup H_{z''})|\\
{}&{}\geq 
r_{z'}+r_{z''}-|I|\\
{}&{} > 
r_{z'}+r_{z''}-r,
\end{align*}
contradicting \ref{eq:h1}. Therefore there exists a hyperedge, say $H_i$, such that $Z-I\subseteq H_i$ and $|I\cap H_i|=r_i$. Thus we get $|I|=|I\cap H_i|+|I-H_i|=r_i+|Z-H_i|$, implying that the cardinality of $I$ depends only on $Z$. Therefore the third independence axiom holds, and the rank formula is also verified.  
If \ref{eq:h2} holds, then the rank formula implies $r_M(S)=\min\big\{r,|S|,\min_{1\leq i\leq q}\{|S-H_i|+r_i\}\big\}=r$, concluding the proof of the theorem.
\end{proof}

We call the matroids that can be obtained in the form provided by Theorem~\ref{thm:hyp} \textbf{elementary split matroids}. When \ref{eq:h2} fails for some $1\leq i\leq q$, the rank of the matroid is less than $r$ by the rank formula. In such a case, replacing $r$ with $r'=\min_{1\leq i \leq q}\{|S-H_i|+r_i\}$ does not violate \ref{eq:h1} while $\cI$ remains the same. Thus a rank-$r$ elementary split matroid can be represented by a hypergraph $\cH=\{H_1,\dots,H_q\}$ and values $r,r_1,\dots,r_q$ satisfying both \ref{eq:h1} and \ref{eq:h2}. It is not difficult to see that the underlying hypergraph can be chosen in such a way that 
\begin{align}
r_i \le r-1\  \text{for $i=1,\dots, q$,} \tag*{(H3)}\label{eq:h3}\\
|H_i| \ge r_i+1\  \text{for $i=1,\dots, q$.} \tag*{(H4)}\label{eq:h4}
\end{align}
Indeed, if a pair $(H_i,r_i)$ violates \ref{eq:h3} or \ref{eq:h4}, then the corresponding constraint $|X\cap H_i|\leq r_i$ is redundant. Therefore, we call the representation \textbf{non-redundant} if all of \ref{eq:h1}--\ref{eq:h4} hold.

Elementary split matroids generalize paving matroids. Indeed, paving matroids correspond to the special case when $r_i=r-1$ for $i=1,\dots,q$. If, in addition, $|H_i|=r$ holds for $i=1,\dots,q$, then we get back the class of sparse paving matroids.

\begin{rem}
The definition of elementary split matroids is closely related to the construction of matroids by cyclic flats, described in Proposition~\ref{prop:cyclic}. Consider a non-redundant hypergraph representation $\cH=\{H_1,\dots,H_q\}$, $r, r_1, \dots, r_q$ of a rank-$r$ elementary split matroid. In order to exclude extreme cases, assume that $q\geq 1$, the $r_i$ values are strictly positive, and \ref{eq:h2} holds with strict inequality for $1\leq i\leq q$. We claim that the family $\mathcal Z = \{\emptyset, H_1, \dots, H_q, S\}$ satisfies the conditions of Proposition~\ref{prop:cyclic} with $r_M(\emptyset) = 0$, $r_M(H_i) = r_i$ and $r_M(S)=r$. Indeed, for different indices $i$ and $j$ we have $|H_i \cap H_j| \le r_i + r_j - r \le (|H_i|-1)-1 = |H_i|-2$ by \ref{eq:h1}, \ref{eq:h3} and \ref{eq:h4}, hence $H_i \not \subseteq H_j$. Thus condition \ref{eq:Z0} is satisfied and $H_i \wedge H_j = \emptyset$, $H_i \vee H_j = S$ for each $i\ne j$. Condition \ref{eq:Z1} holds by $r_M(\emptyset)=0$. Condition \ref{eq:Z2} for $X=\emptyset$ and $Y=H_i$ translates to $0 < r_i < |H_i|$, for $X=\emptyset$ and $Y=S$ it translates to $0<r<|S|$, and for $X=H_i$ and $Y=S$ it translates to $0<r-r_i < |S-H_i|$, all of which are satisfied by our assumptions. Conditions \ref{eq:Z0}--\ref{eq:Z2} imply that \ref{eq:Z3} is satisfied if either $X$ or $Y$ is $0_\mathcal{Z}$ or $1_\mathcal{Z}$, or if $X=Y$. If $X=H_i$ and $Y=H_j$ for $i\ne j$, then \ref{eq:Z3} is equivalent to $r_i + r_j \ge r + |H_i \cap H_j|$, that is, to \ref{eq:h1}. Therefore, Proposition~\ref{prop:cyclic} provides another proof for $M$ being a rank-$r$ matroid whose system of cyclic flats is $\mathcal{Z}$. However, the addition of the missing extreme cases ensures that our class is minor-closed.
\end{rem}

A nice feature of the class of split matroids is that it is closed under duality and taking minors. We show that the same holds for elementary split matroids. In addition, the class of elementary split matroids is closed for truncation, a property that split matroids do not satisfy in general. To see the latter, consider the matroid $M=(U_{1,2}\oplus U_{1,2}\oplus U_{1,2}\oplus U_{1,2})_3$, that is, the $3$-truncation of the direct sum of four rank-$1$ uniform matroids on 2 elements. Then it is not difficult to check that $M$ is connected and has a $U_{0,1}\oplus U_{1,2}\oplus U_{1,1}$-minor, therefore it is not a split matroid, see Theorem~\ref{thm:equiv} later.

\begin{thm} \label{thm:closed}
The class of elementary split matroids is closed under duality, taking minors, and truncation.
\end{thm}
\begin{proof}
Let $M=(S,\cI)$ be a rank-$r$ elementary split matroid and $\cH=\{H_1,\dots,H_q\}$, $r,r_1,\dots,r_q$ be a representation satisfying \ref{eq:h1} and \ref{eq:h2}. For a non-negative integer $k<r$, the $k$-truncation of $M$ is the matroid $(M)_k=(S,\cI_k)$ where $\cI_k=\{X\subseteq S\mid |X|\le k,\ |X\cap H_i|\leq r_i\ \text{for $1\leq i \leq q$}\}$. As $|H_i\cap H_j|\leq r_i+r_j-r\leq r_i+r_j-k$, $(M)_k$ is an elementary split matroid.

Now consider a set $Z\subseteq S$. The deletion of $Z$ results in a matroid $M\backslash Z=(S-Z,\cI_{S-Z})$ where $\cI_{S-Z}=\{X\subseteq S-Z\mid |X|\le r,\ |X\cap (H_i-Z)|\leq r_i\ \text{for $1\leq i \leq q$}\}$. As $|(H_i\cap H_j)- Z|\leq |H_i\cap H_j|\leq r_i+r_j-r$, $M\backslash Z$ is an elementary split matroid. Note that \ref{eq:h2} might not hold for the restriction as the size of the ground set decreased, hence the rank of $M\backslash Z$ might be smaller than $r$.

Finally, define $\overline{H}_i:=S-H_i$, $\overline{r}:=|S|-r$, and $\overline{r}_i:=|\overline{H}_i|-r+r_i$ for $i=1,\dots,q$. Then $\overline{r}\leq |S|$ and $\overline{r},\overline{r}_1,\dots,\overline{r}_q$ are non-negative by $r\leq |S|$ and \ref{eq:h2}. By \ref{eq:h1}, for $1\leq i<j\leq q$, we obtain
\begin{align*}
    |\overline H_i\cap \overline H_j|
    {}&{}=|S|-|H_i|-|H_j|+|H_i\cap H_j|\\
    {}&{}\leq |S|-|H_i|-|H_j|+r_i+r_j-r\\
    {}&{}= \left(|\overline H_i|-r+r_i\right)+\left(|\overline H_j|-r+r_j\right)-\left(|S|-r\right)\\
    {}&{}=\overline r_i+\overline r_j-\overline r.
\end{align*}
By $r_i\geq 0$, for $i=1,\dots,q$, we obtain
\begin{equation*}
    |S-\overline H_i|+\overline r_i
    =|H_i|+|S|-|H_i|-r+r_i\geq \overline{r}.
\end{equation*}
Therefore $\overline\cH=\{\overline H_1,\dots,\overline H_q\}$, $\overline r,\overline r_1,\dots,\overline r_q$ satisfies all the conditions of Theorem~\ref{thm:hyp}, hence $\{X\subseteq S\mid |X|\leq \overline r,\ |X\cap \overline H_i|\leq \overline r_i\ \text{for $i=1,\dots,q$}\}$ forms the independent sets of a rank-$\overline r$ elementary split matroid $\overline M$. For a set $X\subseteq S$ of size $r$, $|X\cap H_i|\leq r_i$ holds if and only if $|\overline X\cap \overline H_i|\leq \overline r_i$ holds, where $\overline X=S-X$. That is, the bases of $M$ are exactly the complements of the bases of $\overline M$, thus $\overline M$ coincides with the dual $M^*$ of $M$. 

As every minor of a matroid can be obtained by a series of deletions and contractions, and $M/Z=(M^*\backslash Z)^*$, the minor-closedness of the class of elementary split matroids follows.
\end{proof}

\begin{rem}
Assume that the representation of $M$ is non-redundant, that is, all of \ref{eq:h1}--\ref{eq:h4} are satisfied. By \ref{eq:h4}, for $i=1,\dots,q$, we obtain
\begin{equation*}
    \overline r_i
    =|\overline H_i|-r+r_i
    =|S|-|H_i|-r+r_i
    \leq \overline r-1.
\end{equation*}
Furthermore, by \ref{eq:h3}, for $i=1,\dots,q$, we obtain
\begin{equation*}
    |\overline H_i|\geq |\overline H_i|-r+r_i+1=\overline r_i+1.
\end{equation*}
That is, $\overline\cH=\{\overline H_1,\dots,\overline H_q\}$, $\overline r,\overline r_1,\dots,\overline r_q$ satisfies \ref{eq:h3} and \ref{eq:h4} as well, hence then the representation of the dual provided by the proof of Theorem~\ref{thm:closed} is also non-redundant.
\end{rem}

The following observation will be helpful when characterizing binary split matroids.

\begin{thm} \label{thm:uniform}
Consider a non-redundant representation $\cH=\{H_1,\dots, H_q\}, r,r_1, \dots, r_q$ of an elementary split matroid $M$ on ground set $S$. Then $M|H_i \cong U_{r_i, |H_i|}$ and $M/H_i \cong U_{r-r_i, |S-H_i|}$ for $i=1,\dots, q$.
\end{thm}
\begin{proof}
Let $X \subseteq H_i$ be a subset of size $r_i$ for some $1\leq i \leq q$. Then $|X|=r_i < r$ by \ref{eq:h3} and $|X \cap H_j| \le |H_i \cap H_j| \le r_i+r_j-r < r_j$ by \ref{eq:h1} and \ref{eq:h3} for each index $j \ne i$, hence $X$ is independent in $M$.  As each independent subset of $H_i$ has size at most $r_i$, we get that $M|H_i \cong U_{r_i, |H_i|}$. Considering the hypergraph representation of the dual matroid $M^*$ constructed in the proof of Theorem~\ref{thm:closed}, it follows that $M^*|\overline{H}_i \cong U_{\overline{r}_i, |\overline{H}_i|} = U_{|\overline{H}_i|-r+r_i, |\overline{H}_i|}$, hence $M / H_i = (M^*|\overline{H}_i)^* \cong U_{r-r_i, |\overline{H}_i|}$. 
\end{proof}

\section{Excluded-minor characterization} 
\label{sec:minor}

The aim of this section is to give an excluded-minor characterization of elementary split matroids. In contrast to split matroids where five forbidden minors are needed, elementary split matroids can be characterized by a single one. The next theorem determines the unique forbidden minor, and establishes a connection between elementary and connected split matroids.

\begin{thm} \label{thm:equiv}
The following are equivalent for a matroid $M$ on ground set $S$.
\begin{enumerate}[label={(\roman*)}]\itemsep0em
    \item $M$ is an elementary split matroid. \label{it:eq1}
    \item $M$ has no $U_{0,1} \oplus U_{1,2} \oplus U_{1,1}$-minor. \label{it:eq2}
    \item $M$ is a loopless and coloopless matroid whose proper cyclic flats form a clutter, or $M$ is the direct sum of a uniform matroid with either a rank-0 matroid or a free matroid. \label{it:eq3}
    \item $M$ is a connected split matroid or the direct sum of two uniform matroids. \label{it:eq4}
\end{enumerate}
\end{thm}
\begin{proof}

\ref{it:eq1} $\Rightarrow$ \ref{it:eq2}  The class of elementary split matroids is closed under taking minors by Theorem~\ref{thm:closed}, hence it suffices to show that $M=U_{0,1}\oplus U_{1,2}\oplus U_{1,1}$ is not an elementary split matroid. Suppose to the contrary that there exists a hypergraph $\cH=\{H_1, \dots, H_q\}$ and values $r_1, \dots, r_q$ satisfying \ref{eq:h1}--\ref{eq:h4} with $r=2$ which define $M$. As $M$ has exactly one loop, there is an index $i$ such that $|H_i|=1$ and $r_i=0$. We claim that $H_i$ is the unique hyperedge in $\cH$. Indeed, for an arbitrary index $j \ne i$, we have $0\leq |H_i \cap H_j|\le r_j-r < 0$ which is not possible. Hence $i=q=1$ and $M \cong U_{0,1}\oplus U_{2,3}$, a contradiction.
\medskip

\noindent \ref{it:eq2} $\Rightarrow$ \ref{it:eq3} Suppose that $M$ is $U_{0,1}\oplus U_{1,2}\oplus U_{1,1}$-minor-free and has proper cyclic flats $X$ and $Y$ such that $X \subsetneq Y$. As $X$ is a flat and $Y$ is cyclic, $(M|Y)/X$ is loopless and not free, hence it has a $U_{1,2}$-minor. Let $x \in X$ and $z \in S-Y$ and consider the matroid $M' = (M|(Y+z))/(X-x)$. As $X$ is cyclic, $x$ is a loop in $M/(X-x)$, and it is also a loop in $M'$. As $Y$ is a flat, $z$ is a coloop in $M|(Y+z)$, hence it is a coloop in $M'$ as well.  We get that $M' \cong U_{0,1} \oplus (M|Y)/X \oplus U_{1,1}$ where $(M|Y)/X$ has a $U_{1,2}$-minor. Therefore $M'$ has a $U_{0,1} \oplus U_{1,2} \oplus U_{1,1}$-minor, hence so does $M$. This contradiction proves that the proper cyclic flats of $M$ form a clutter.

It remains to consider the case when $M$ is $U_{0,1}\oplus U_{1,2} \oplus U_{1,1}$-minor-free and has a loop or a coloop. By duality, we may assume that $M$ contains a loop. Let $M'$ be the matroid obtained by deleting all the loops and coloops from $M$. If $M'$ is empty, then $M$ is the direct sum of a rank-0 and a (possibly empty) free matroid. Otherwise, as $M'$ is loopless and coloopless, it has a $U_{1,2}$-minor. This implies that $M$ is coloopless, as otherwise it contains a $U_{0,1}\oplus U_{1,2}\oplus U_{1,1}$-minor, contradicting the assumption. We also get that $M'$ is connected, as otherwise it has a $U_{1,2}\oplus U_{1,2}$-minor, meaning that $M$ contains $U_{0,1}\oplus U_{1,2}\oplus U_{1,2}$ as a minor. By Proposition~\ref{prop:w2}, each connected non-uniform matroid contains $M(\mathcal{W}_2)$ as a minor. However, $M(\mathcal{W}_2)$ has a $U_{1,2}\oplus U_{1,1}$-minor, hence $M'$ is necessarily a uniform matroid. Therefore $M$ is the direct sum of a rank-0 and a uniform matroid.
\medskip

\ref{it:eq3} $\Rightarrow$ \ref{it:eq4} The implication is immediate if $M$ is the direct sum of a uniform matroid with either a rank-$0$ matroid or a free matroid. Hence we may assume that $M$ is loopless, coloopless and its cyclic flats form a clutter. If $M$ is connected, we need to show that any two distinct split flacets $F$ and $G$ are compatible, that is, $|F\cap G| +r_M(S) \le r_M(F)+r_M(G)$ by Proposition~\ref{prop:flacet}. As proper cyclic flats form a clutter and each split flacet is a proper cyclic flat, this inequality is equivalent to the cyclic flat axiom \ref{eq:Z3} for $X=F$ and $Y=G$.

Consider the case when $M$ is disconnected and let $Z_1,\dots, Z_t$ denote its connected components. As $M$ is loopless and coloopless, $Z_1, \dots, Z_t$ are proper cyclic flats. This implies $t=2$, as otherwise $Z_1$ and $Z_1 \cup Z_2$ are both proper cyclic flats with $Z_1 \subsetneq Z_1 \cup Z_2$, contradicting the assumption that proper cyclic flats form a clutter. Assume that there is a proper cyclic flat $Z\not\in\{Z_1, Z_2\}$. Then in the lattice of cyclic flats $Z \wedge Z_i = \emptyset$ and $Z\vee Z_i = S$, hence \ref{eq:Z3} implies that $r_M(Z)+r_M(Z_i) \ge r_M(S) + |Z \cap Z_i|$ holds for $i=1,2$. Thus we get
\begin{align*} 
r_M(Z)+1 
{}&{}\le |Z| \\
{}&{} = |Z \cap Z_1| + |Z \cap Z_2| \\
{}&{} \le (r_M(Z)+r_M(Z_1) -r_M(S)) + (r_M(Z)+r_M(Z_2)-r_M(S)) \\ 
{}&{} = 2r_M(Z)-r_M(S) \\
{}&{}\le r_M(Z)-1, 
\end{align*}
a contradiction. Therefore, the only proper cyclic flats of $M$ are $Z_1$ and $Z_2$, hence $M|Z_1$ and $M|Z_2$ are uniform matroids. We proved that $M$ is the direct sum of two uniform matroids.
\medskip

\noindent \ref{it:eq4} $\Rightarrow$ \ref{it:eq1} Assume first that $M$ is the direct sum of a rank-$r_1$ uniform matroid on ground set $H_1$ and a rank-$r_2$ uniform matroid on ground set $H_2$. Let $r:=r_1+r_2$. Then $M$ is the elementary split matroid on ground set $S:=H_1\cup H_2$ corresponding to the hypergraph $\cH=\{H_1,H_2\}$ and non-negative integers $r,r_1,r_2$. Indeed, $r\leq |S|$ holds and \ref{eq:h1} is satisfied as $0=|H_1\cap H_2|\leq r_1+r_2-r=0$.

Now consider the case when $M$ is a connected split matroid. Let $r$ denote the rank of $M$, $\cH=\{H_1,\dots,H_q\}$ be the collection of split flacets, and set the value of $r_i$ to be the rank of $H_i$ for $i=1,\dots,q$. Then $r\leq |S|$ and the values $r,r_1,\dots,r_q$ are non-negative. As $M$ is a split matroid, any two split flacets are compatible, therefore \ref{eq:h1} is satisfied. By Theorem~\ref{thm:hyp}, $\cH=\{H_1,\dots,H_q\}$, $r,r_1,\dots,r_q$ define an elementary split matroid $\hat M$. We claim that $\hat M$ is identical to $M$. Indeed, if a set $B\subseteq S$ is a basis in $M$ then, by definition, it is also a basis in $\hat M$. If $B$ is not a basis in $M$, then the charasteristic vector of $B$ is not contained in the base polytope of $M$. As the base polytope is completely determined by the flacet inequalities (see e.g. \cite[Theorem 2.6]{oh2021facets}), this means that $|B\cap F|>r_M(F)$ for some flacet $F$ of $M$. Since $M$ is connected, necessarily $|F|\geq 2$. By Proposition~\ref{prop:flacet}\ref{it:2}, $F$ is a split flacet, hence $B$ is not a basis of $\hat M$ either.
\end{proof}

\section{Applications}
\label{sec:app}

As an application of our results, we give a new proof for the result of Cameron and Mayhew~\cite{cameron2021excluded}. Furthermore, we further give a complete list of binary split matroids.

\subsection{Split matroids}

Based on the previous results, we give a different and shorter proof of the excluded minor characterization of split matroids 
originally proved by Cameron and Mayhew~\cite{cameron2021excluded}. As already observed in \cite{joswig2017matroids}, the only disconnected excluded minor for the class of split matroids is $M(\mathcal{W}_2) \oplus M(\mathcal{W}_2)$. This follows from Proposition~\ref{prop:w2} and the fact that a disconnected matroid is a split matroid if and only if it is the direct sum of a connected split matroid and uniform matroids (see \cite[Proposition~2.7]{cameron2021excluded}). Joswig and Schröter \cite{joswig2017matroids} also identified four connected rank-3 excluded minors on 6 elements, these matroids $S_1, S_2, S_3, S_4$ are given by their geometric representations on Figure~\ref{fig:S}, see also \cite{cameron2021excluded}. 

\begin{thm} 
The only connected excluded minors for split matroids are $S_1, S_2, S_3$ and $S_4$.
\end{thm}

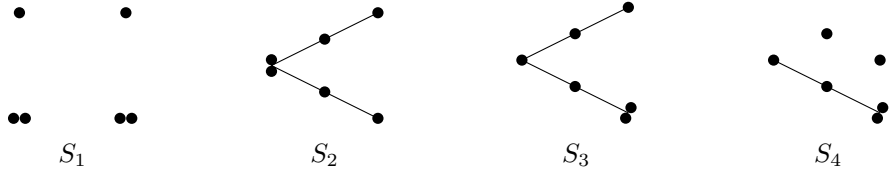
\begin{figure}[t!]
\centering
\begin{subfigure}[t]{0.2\textwidth}
    \centering
    \begin{tikzpicture}[scale=0.7]
        \fill (-0.11,0) circle (3pt)
            (0.11,0) circle (3pt)
            (1.89,0) circle (3pt)
            (2.11,0) circle (3pt)
            (0, 2) circle (3pt)
            (2,2) circle (3pt);
    \end{tikzpicture}
    \caption*{$S_1$}
\end{subfigure} 
\begin{subfigure}[t]{0.2\textwidth}
    \centering
    \begin{tikzpicture}[scale=0.7]
        \fill (0,1.11) circle (3pt)
            (0,0.89) circle (3pt)
            (1,0.5) circle (3pt)
            (1,1.5) circle (3pt)
            (2, 0) circle (3pt)
            (2,2) circle (3pt);
        \draw (0,1) -- (2,0)
            (0,1) -- (2,2);
    \end{tikzpicture}
    \caption*{$S_2$}
\end{subfigure} 
\begin{subfigure}[t]{0.2\textwidth}
    \centering
    \begin{tikzpicture}[scale=0.7]
        \fill (0,1) circle (3pt)
            (1,0.5) circle (3pt)
            (1,1.5) circle (3pt)
            (1.95, -0.1) circle (3pt)
            (2.05, 0.1) circle (3pt)
            (2,2) circle (3pt);
        \draw (0,1) -- (2,0)
            (0,1) -- (2,2);
    \end{tikzpicture}
    \caption*{$S_3$}
\end{subfigure} 
\begin{subfigure}[t]{0.2\textwidth}
    \centering
    \begin{tikzpicture}[scale=0.7]
        \fill (0,1) circle (3pt)
            (1,0.5) circle (3pt)
            (1,1.5) circle (3pt)
            (1.95, -0.1) circle (3pt)
            (2.05, 0.1) circle (3pt)
            (2,1) circle (3pt);
        \draw (0,1) -- (2,0);
    \end{tikzpicture}
    \caption*{$S_4$}
\end{subfigure} \hfill
\caption{The connected excluded minors for split matroids \cite{cameron2021excluded}.}
\label{fig:S}
\end{figure}

\begin{proof}
Let $M$ be a connected matroid which is not a split matroid but each of its proper minors is. By Theorem~\ref{thm:equiv}, connected split matroids coincide with connected elementary split matroids, hence $M$ is not elementary split while it is nearly-elementary-split by Proposition~\ref{prop:connected}. As $M$ is not elementary split, $(M/X)\backslash Y \cong U_{0,1}\oplus U_{1,2}\oplus U_{1,1}$ for some subsets $X, Y\subseteq S$ by Theorem~\ref{thm:equiv}. By $M$ being loopless and coloopless, the sets $X$ and $Y$ are nonempty. Let $X'\subseteq X$ be a nonempty subset and consider the matroid $N=M/X'$. Notice that $N$ is disconnected since it is split but not elementary split. Moreover, $N$ is coloopless since $M$ is coloopless.

We claim that $N$ contains exactly one loop. If $N$ is loopless, then each of its connected components has a $U_{1,2}$-minor and each of its non-uniform components has an $M(\mathcal{W}_2)$-minor by Proposition~\ref{prop:w2}. As $N$ is not the direct sum of two uniform matroids by Theorem~\ref{thm:equiv}, we get that it either contains $M(\mathcal{W}_2)\oplus U_{1,2}$ or $U_{1,2}\oplus U_{1,2}\oplus U_{1,2}$ as a minor. This contradicts Proposition~\ref{prop:nearly} since neither of these two matroids is nearly-elementary-split. Similarly, if $N$ has at least two loops, then it has a non-uniform connected component or at least two loopless connected components. Thus $N$ contains either $U_{0,2}\oplus M(\mathcal{W}_2)$ or $U_{0,2}\oplus U_{1,2}\oplus U_{1,2}$ as a minor, both of which contradict Proposition~\ref{prop:nearly}. This proves that $N$ has exactly one loop.

Suppose that $|X|\ge 2$ and choose distinct elements $x_1, x_2 \in X$. By our previous observation, $M/x_1$ contains exactly one loop $l_1$ and $M/x_2$ contains exactly one loop $l_2$, that is, $\{x_1, l_1\}$ and $\{x_2, l_2\}$ are parallel classes of $M$. If $x_2=l_1$, then $M/\{x_1, x_2\}$ is loopless. Otherwise, parallel classes $\{x_1, l_1\}$ and $\{x_2, l_2\}$ are disjoint and $M/\{x_1, x_2\}$ contains $l_1$ and $l_2$ as loops. Both of these cases contradict that $M/\{x_1, x_2\}$ contains exactly one loop. We conclude that $|X|=1$. The class of split matroids is closed under duality, hence the dual $M^*$ is also an excluded minor. Applying the previous argument to $(M^*\backslash X) / Y =((M/X)\backslash Y)^*= U_{1,1}\oplus U_{1,2} \oplus U_{0,1}$, we get that $|Y|=1$ holds as well. 

We proved that $M$ is a rank-3 matroid on 6 elements. Denote the element of $X$ by $a$, the loop of $M/a$ by $b$, the element of $Y$ by $e$, the coloop of $M\backslash e$ by $f$, and the remaining two elements of the ground set by $c$ and $d$. Then $\{e,f\}$ is a cocircuit of $M$, thus $M|\{a,b,c,d\}$ is a loopless rank-2 matroid containing the parallel class $\{a,b\}$, hence it is isomorphic to either $U_{1,2}\oplus U_{1,2}$ or $M(\mathcal{W}_2)$. The former case gives the matroid $S_1$. In the latter case consider the lines $\{a,b,c,d\}$ and $\clo_M(\{e,f\})$. If $\clo_M(\{e,f\}) = \{e,f\}$, we get $S_4$. Otherwise, the intersection of $\{a,b,c,d\}$ and $\clo_M(\{e,f\})$ is a rank-1 flat, thus it is $\{a,b\}$, $\{c\}$ or $\{d\}$. The first case gives $S_2$ and the latter two cases give $S_3$.
\end{proof}

\subsection{Binary split matroids}
\label{sec:binary}

Acketa~\cite{acketa1988binary} gave a complete list of binary paving matroids:  $U_{r,n}$ for $r\in \{0,1,n-1,n\}$, loopless rank-2 matroids with at most three parallel classes, $M(K_4-e)$, $M(K_4)$, $M(K_{2,3})$, $F_7$, $F_7^*$ and $AG(3,2)$ (see also \cite{oxley2011matroid} for the definition of the latter three matroids). Based on this and our previous results, we extend this list to contain all binary split matroids. As each split matroid has at most one non-uniform connected component, we only consider the connected case. Recall that the only forbidden minor for binary matroids is $U_{2,4}$ by Tutte~\cite{tutte1958homotopyII}.

\begin{thm} The following is a complete list of connected binary split matroids on at least two elements.
\begin{enumerate}[label={(\alph*)}] \itemsep0em
    \item Matroids obtained by adding (possibly zero) parallel copies to an element of $U_{r-1,r}$ for any $r\ge 2$. \label{eq:bin1}
    \item Loopless rank-2 matroids with exactly three parallel classes, and their duals. \label{eq:bin2}
    \item Connected binary (sparse) paving matroids of rank and corank at least three: $M(K_4)$, $F_7$, $F_7^*$, $AG(3,2)$. \label{eq:bin3}
\end{enumerate}
\end{thm}
\begin{proof}
It is not difficult to check that all the listed matroids are connected binary split matroids. It remains to prove that each binary split matroid $M$ is included in the list. Connected matroids of rank or corank one are $U_{1,n}$ and $U_{n-1,n}$, and these are included in \ref{eq:bin1}. Connected binary matroids of rank or corank two are exactly the matroids listed as \ref{eq:bin2}. Thus we may assume that $M$ has rank and corank at least three. As \ref{eq:bin3} contains binary paving matroids of rank and corank at least three from the list of Acketa~\cite{acketa1988binary}, it only remains to consider the non-paving case.

Let $M$ be given by a non-redundant hypergraph representation $\cH = \{H_1, \dots, H_q\}$, $r, r_1, \allowbreak\dots, r_q$. As $M$ is non-paving, $r_i \le r-2$ holds for some index $i$. By Theorem~\ref{thm:uniform}, $M/H_i \cong U_{r-r_i, |S-H_i|}$ where $2\le r-r_i$ by our assumption and $r-r_i+1 \le |S-H_i|$ as $M$ is coloopless. As $M$ contains no $U_{2,4}$-minor, nececssarily $|S-H_i|=r-r_i+1$. Then $|H_i| = |S|-r+r_i-1 \ge r_i+2$ as $M$ has corank at least 3. Furthermore, $M|H_i \cong U_{r_i, |H_i|}$ contains no $U_{2,4}$-minor, implying $r_i=1$ and $|S-H_i| = r$. Suppose that $q\ge 2$ and pick an index $j \ne i$. Applying \ref{eq:h1} and \ref{eq:h3}, we get $1+r_j-r \ge |H_i \cap H_j| \ge |H_j| - |S-H_i| \ge r_j + 1 - r$, hence $S-H_i \subseteq H_j$ and $|H_j| = r_j+1$. Then  $S-H_i\subseteq H_j$ implies that $r=|S-H_i| \le |H_j| = r_j+1 \le r$, hence $S-H_j = H_j$ and $r_j = r-1$. Therefore, $S$ is the disjoint union of the rank-1 set $H_i$ and the rank-$(r-1)$ set $H_j$, contradicting the connectivity of $M$. This proves that $q=1$, so $M\cong (U_{1, |H_i|} \oplus U_{r, r})_r$ as described in \ref{eq:bin1}.
\end{proof}

\section*{Acknowledgement} The work was supported by the Lend\"ulet Programme of the Hungarian Academy of Sciences -- grant number LP2021-1/2021 and by the Hungarian National Research, Development and Innovation Office -- NKFIH, grant numbers FK128673 and TKP2020-NKA-06.
Yutaro Yamaguchi was supported by JSPS KAKENHI Grant Numbers JP20K19743 and JP20H00605 and by Overseas Research Program in Graduate School of Information Science and Technology, Osaka University.
Yu Yokoi was supported by JSPS KAKENHI Grant Number JP18K18004 and JST PRESTO Grant Number JPMJPR212B.

\bibliographystyle{abbrv}
\bibliography{split}

\end{document}